\documentclass[12pt]{article}
\date{}

\usepackage[utf8]{inputenc}
\usepackage[english]{babel}
\usepackage{amssymb}
\usepackage{amsfonts}
\usepackage{amsthm}
\usepackage{mathrsfs}
\usepackage{enumerate}
\usepackage{amsmath}
\usepackage{color}
\usepackage{hyperref}

\usepackage[normalem]{ulem}
\usepackage{titlesec}
\titleformat*{\subsection}{\fontsize{13}{15}\selectfont}

\newcommand{\F}{\mathbb{F}}

\newcommand{\Supp}{\textnormal{supp}}

\newtheorem{theorem}{Theorem}[section]
\newtheorem{corollary}[theorem]{Corollary}
\newtheorem{lemma}[theorem]{Lemma}

\theoremstyle{definition}
\newtheorem*{definition}{Definition}
\newtheorem{example}{Example}

\title{Minimal Codes From Characteristic Functions Not Satisfying The Ashikhmin-Barg Condition}
\author{Julien Sorci}

\begin{document}
\maketitle

\abstract{A \textit{minimal code} is a linear code where the only instance that a codeword has its support contained in the support of another codeword is when the codewords are scalar multiples of each other. Ashikhmin and Barg gave a sufficient condition for a code to be minimal, which led to much interest in constructing minimal codes that do not satisfy their condition. We consider a particular family of codes $\mathcal C_f$ when $f$ is the indicator function of a set of points, and prove a sufficient condition for $\mathcal C_f$ to be minimal and not satisfy Ashikhmin and Barg's condition based on certain geometric properties of the support of $f$. We give a lower bound on the size of a set of points satisfying these geometric properties and show that the bound is tight.}

\section{Introduction}

	Linear codes have found applications in areas far beyond error-correction. For example, an $(S,T)$ \textit{secret-sharing scheme} is a collection of $S$ ``shares" of a $q$-ary secret such that knowledge of any $T$ shares determines the secret, but knowledge of $T-1$ or fewer shares gives no information. Massey showed that certain codewords in the dual of a linear code, called \textit{minimal codewords}, can be used to construct a secret sharing scheme \cite{M}. However, determining the set of minimal codewords in a linear code is a hard problem in general, which galvanized the search for codes where every codeword is minimal, referred to as \textit{minimal codes}. 
		
	In 1998, Ashikhmin and Barg gave a sufficient condition for a code to be minimal based on the ratio of the maximum and minimum weight of the code \cite{AB}. 
\begin{lemma}
	A $q$-ary linear $[n,k]$ code $\mathcal C$ is minimal if 
\[ \frac{w_{\max}}{w_{\min}} < \frac{q}{q-1},    \]
	where $w_{\min}$ and $w_{\max}$ are the minimum and maximum weights of $\mathcal C$, respectively. 
\end{lemma} 
	For many years following this result all known examples of minimal codes satisfied Ashikhmin and Barg's condition, until 2017 when Chang and Hyun \cite{CH} constructed a minimal binary code from a simplicial complex that did not. As for finding a family of $q$-ary minimal codes not satisfying Ashikhmin and Barg's condition for $q$ an odd prime power, a particular code that has received much attention is the code $\mathcal C_f$, which we will now define. Let $f: \F_q^n \rightarrow \F_q$ be an arbitrary but fixed function. We define $\mathcal C_f$ to be the $\F_q$-subspace of $\F_q^{q^n-1}$ spanned by vectors of the form
\begin{equation*}
	c(u,v):= (uf(x)+v \cdot x)_{x \in \F_q^n \setminus \{0\}}
\end{equation*}
	where $u \in \F_q$, $v \in \F_q^n$, and $v\cdot x$ is the usual dot product. We record below the basic parameters of the code, which are well known. 
\begin{lemma}
	If $f:\F_q^n \rightarrow \F_q$ is not linear and $f(a) \neq 0$ for some $a \in \F_q^n \setminus \{0 \}$, then $\mathcal C_f$ is a $[q^n-1, n+1]$ linear code. 
\end{lemma}
	In \cite{DHZ}, Ding et. al. gave necessary and sufficient conditions for the code $\mathcal C_f$ to be minimal when $q=2$ based on the Walsh-Hadamard transform of $f$, which is defined to be the function $\hat{f}(x):= \sum_{v \in \F_2^n}(-1)^{f(x)+v\cdot x}$.
\begin{lemma}\label{dingminimal}
	If $f:\F_2^n \rightarrow \F_2$ is not linear and $f(a) \neq 0$ for some $a \in \F_2^n \setminus \{0 \}$ then the binary code $\mathcal C_f$ is minimal if and only if $\hat{f}(x) + \hat{f}(y) \neq 2^n$ and $\hat{f}(x)-\hat{f}(y) \neq 2^n$ for every pair of distinct vectors $x, y \in \F_2^n$. 
\end{lemma}
	A boolean function $f: \F_2^n \rightarrow \F_2$ with $n$ a positive even integer is said to be \textit{bent} if $|\hat{f}(x)|=2^{n/2}$ for all $x \in \F_2^n$. Using Lemma~\ref{dingminimal}, it is easy to see that the binary code $\mathcal C_f$ is minimal when $f$ is a bent function. 
	
	Bonini et. al. studied the code $\mathcal C_f$ for arbitrary prime powers $q$ and functions $f$, and showed that if the zero set of $f$ satisfies certain geometric properties then $\mathcal C_f$ is a minimal code \cite{BB}. We continue this line of work by considering the code $\mathcal C_f$ when $f$ is the indicator function of a set, and give sufficient conditions for $\mathcal C_f$ to be minimal and not satisfying the condition of Ashikhmin and Barg in terms of the geometric properties of the support of $f$. We give a tight lower bound on the size of sets satisfying our geometric conditions, and give an explicit example of a set meeting the lower bound. In section 2 we lay out the notation used, and in section 3 we present the main results.

\section{Notation}

\begin{definition}
	
	A \textit{linear} $[n,k]$ \textit{code} is a $k$-dimensional subspace $\mathcal C$ of $\F_q^n$.	
\end{definition}

\begin{definition}
	
	A codeword $c \in \mathcal C$ is said to be \textit{minimal} if whenever $\Supp(c) \subseteq \Supp(c^\prime)$ for some codeword $c^\prime \in \mathcal C$ it implies that $c = \lambda c^\prime$ for some $\lambda \in \F_q^\times$. A linear code $\mathcal C$ is \textit{minimal} if all codewords of $\mathcal C$ are minimal. 
	
\end{definition}

We summarize some of the notation used in the paper:

\begin{itemize}
	\item $e_i$ will denote the $i^{th}$ standard basis vector. 
	\item For $v \in \F_q^n$, we will let $H(v)$ denote the set $\{u \in \F_q^n : v \cdot u =0 \}$.
	\item For a function $f : \F_q^n \rightarrow \F_q$, we will let $V(f)$ denote the set of zeros of $f$, $\{ u \in \F_q^n : f(u)=0 \}$.
	\item If $U$ is any set, we will let $U^*:=U \setminus \{0 \}$, and $\overline{U}$ will denote the complement of $U$. 
	\item By a \textit{hyperplane}, we will mean an $(n-1)$-dimensional subspace of $\F_q^n$.
	\item By an \textit{affine hyperplane}, we will mean a coset of an $(n-1)$-dimensional subspace of $\F_q^n$. Note that by our convention, a hyperplane is also an affine hyperplane.   
\end{itemize}

\section{The Main Results}

The next theorem is our main result.

\begin{theorem}\label{main}
	
	Let $q$ be an arbitrary prime power, and let $S \subseteq \F_q^n \setminus \{0\}$ be a set of points such that

\begin{enumerate}
	\item $S$ is not contained in any affine hyperplane,
	\item $S$ meets every affine hyperplane, 
	\item $|S| < q^{n-2}(q-1)$
\end{enumerate}
	
	Then $\mathcal C_f$ with $f$ the indicator function of $S$ is a minimal code that does not satisfy the Ashikhmin-Barg condition.
	
\end{theorem}

\begin{proof}
	
	Suppose that $\Supp( c(u^\prime,v^\prime)) \subseteq \Supp(c(u,v))$ for some codewords $c(u,v), c(u^\prime, v^\prime)$ of $\mathcal C_f$. Equivalently, 
\begin{equation}\label{assump}
	V(uf(x)+v \cdot x)^* \subseteq V(u^\prime f(x) + v^\prime \cdot x)^*
\end{equation}
	 We proceed by cases to show that $c(u,v)=\lambda c(u^\prime,v^\prime)$ for some $\lambda \in \F_q^\times$.  

\textbf{Case 1:} If $v=0$, then it implies $V(f)^* \subseteq H(v^\prime)^*$, so that $\overline{H(v^\prime)^*} \subseteq S$, contradicting that $|S| < q^{n-2}(q-1)$. 

\textbf{Case 2:} If $v^\prime =0$ then $V(uf(x)+v\cdot x)^* \subseteq V(f)^*$. From the partition
\begin{equation}\label{partition}
	V(uf(x) + v\cdot x)^* = (V(f)^* \cap H(v)^*) \cup ( S \cap \{v \cdot x = -u \})
\end{equation}
 it implies that $S$ does not meet the affine hyperplane $\{v \cdot x =-u \}$, a contradiction. 

\textbf{Case 3:} If $v, v^\prime \neq 0$, then from equation~\ref{assump} and the partition of equation~\ref{partition} we have
\begin{equation}
	V(f)^* \cap H(v)^* \subseteq V(f)^* \cap H(v^\prime)^* \subseteq H(v^\prime)^*
\end{equation}
	Here $|H(v)^* \cap V(f)^*| = |H(v)^* \setminus S| \geq q^{n-1} - |S| \geq q^{n-2}+1$, so that $H(v)^* \cap V(f)^*$ is not contained in a hyperplane other than $H(v)$, i.e. $H(v)=H(v^\prime)$. Thus we have $v^\prime = \lambda v$ for some $\lambda \in \F_q^\times$. Since $S$ meets every affine hyperplane, we can choose some $y$ in $S \cap \{v\cdot x = -u \}$. Equation~\ref{assump} then implies $u=-v\cdot y$ and $u^\prime=-\lambda v \cdot y$, so that $u^\prime = \lambda u$. We therefore conclude $c(u^\prime, v^\prime) = \lambda c(u,v)$ in this case, as was required.  
	
\textbf{Case 4:} If $u=0$ and $H(v) \neq H(v^\prime)$, then equation~\ref{assump} reads as $H(v)^* \subseteq V(u^\prime f(x)+v^\prime \cdot x)^*$. Using the partition of equation~\ref{partition} and the assumption that $S$ is not contained in any affine hyperplane, we have
\begin{equation*}
	\begin{split}
		q^{n-1}-1 &= |H(v)^*| \\
		&= |V(f)^* \cap H(v^\prime)^* \cap H(v)^*| + |S \cap \{ v^\prime \cdot x =-u^\prime \} \cap H(v)^* | \\
		& \leq |H(v)^* \cap H(v^\prime)^*| + |S \cap \{ v^\prime \cdot x =-u^\prime \} | \\
		& \leq q^{n-2}-1 + q^{n-2}(q-1) - 1 \\
		&= q^{n-1} - 2\\
	\end{split}
\end{equation*}
This clear contradiction means we therefore must have $H(v)=H(v^\prime)$, so that $v^\prime = \lambda v$ for some $\lambda \in \F_q^\times$. But the containment $H(v)^* \subseteq V(u^\prime f(x) + \lambda v \cdot x)^*$ implies that $H(v)^* \subseteq V(f)^*$, or equivalently $S \subseteq \overline{H(v)^*}$. This contradicts that $S$ meets the hyperplane $H(v)$. 

\textbf{Case 5:} If $u^\prime =0$, then we have $V(uf(x)+v\cdot x)^* \subseteq H(v^\prime)^*$, which follows by case 3. 
	
	Lastly we check that the code $\mathcal C_f$ does not satisfy the Ashikhmin-Barg condition. The maximum weight is at least the weight of $c(0,1)$, which is the number of points in $H(v)$, and the minimum weight is at most the weight of $c(1,0)$, which is $|S|$. Thus:
\begin{equation*}
	\frac{w_{\max}}{w_{\min}} \geq \frac{q^{n-1}}{|S|} >  \frac{q}{q-1} 
\end{equation*}

\end{proof}

When $q=2$ the conditions of Theorem~\ref{main} simplify considerably, which we record in the following corollary.

\begin{corollary}\label{mainbinary}

	Let $S \subseteq \F_2^n \setminus \{0 \}$ be a set of points such that
\begin{enumerate}
	\item $S$ is not contained in any hyperplane,
	\item $S$ meets every hyperplane,
	\item $|S| < 2^{n-2}$
\end{enumerate}
Then the binary code $C_f$ with $f$ the indicator function of $S$ is a minimal code not satisfying the Ashikhmin-Barg condition.

\end{corollary}

\begin{proof}
	
	It suffices to check that when $q=2$, and $S$ is a set of points that is not contained in a hyperplane and meets every hyperplane, then $S$ is not contained in an affine hyperplane, and meets every affine hyperplane. 
	
	To see that $S$ is not contained in an affine hyperplane, suppose that $H$ is an affine hyperplane not containing the origin, and that $S \subseteq H$. Since $q=2$, then $\overline{H}$ is a hyperplane, so that $S$ does not meet the hyperplane $\overline{H}$, a contradiction.
	
	Similarly, if $S$ does not meet the affine hyperplane $H$ not containing the origin, then $S$ is contained in $\overline{H}$, which is a hyperplane. Therefore $S$ meets every affine hyperplane, so that $S$ satisfies the conditions of Theorem~\ref{main}. 
	
\end{proof}

\begin{example}
	
	Assume that $n \geq 6$ is an even positive integer, and let $q=2$. A \textit{partial spread} of order $s$ is a set of $n/2$-dimensional subspaces $\{U_1,...,U_s\}$ of $\F_2^n$ such that $U_i \cap U_j = \{ 0 \}$ for all $1 \leq i, j \leq s$. It is easy to see that a partial spread of order $s$ has at most $2^{n/2}+1$ elements. 
	
	In \cite{DHZ}, Ding et. al. showed that when $1 \leq s \leq 2^{n/2}+1$, $s \notin \{1, 2^{n/2}, 2^{n/2}+1 \}$, then $\mathcal C_f$ with $f$ the indicator function of the set $S = \cup_{i=1}^s U_i^*$ is a minimal code. Moreover, they showed that if, in addition, we have $s \leq 2^{\frac{n}{2}-2}$ then $\mathcal C_f$ does not satisfy Ashikhmin and Barg's condition. They proved this by computing the Walsh-Hadamard transform of $f$ and then applying Lemma~\ref{dingminimal}, but we can alternatively check that the set $S$ satisfies the conditions of Corollary~\ref{mainbinary}. 
	
	Since $s \geq 2$, then $S$ clearly spans $\F_2^n$, and the assumption that $n \geq 6$ means that $\dim(U_i) \geq 3$, so $S$ meets every hyperplane. Finally, we have in general that $|S|=s(2^{n/2}-1)$, so if we assume that $s \leq 2^{\frac{n}{2}-2}$ then an easy computation shows that $|S| \leq 2^{n-2}-2^{\frac{n}{2}-2}< 2^{n-2}$. Therefore $S$ indeed satisfies the conditions of Corollary~\ref{mainbinary}. 
		
\end{example}

\begin{example}
	
	Let $n \geq 7$ and $2 \leq k \leq \lfloor \frac{n-3}{2} \rfloor$. Let $S$ be the set of vectors of $\F_2^n$ with weight at most $k$. In \cite{DHZ}, Ding et. al. showed that $\mathcal C_f$ with $f$ the indicator function of $S$ is a minimal $[2^n-1, n+1, \sum_{i=1}^k {n \choose i}]$ binary code, and moreover that $\mathcal C_f$ does not satisfy the Ashikhmin-Barg condition if and only if 
\begin{equation}\label{dingbd}
	1 + 2 \sum_{i=1}^k {n \choose i} \leq 2^{n-1} + {n-1 \choose k}
\end{equation}
	
	We alternatively check that the set $S$ satisfies the conditions of Corollary~\ref{mainbinary}. Since $S$ contains the standard basis vectors, $S$ is clearly not contained in any hyperplane. Given any hyperplane $H(v)$, at least one of the vectors $e_1$, $e_2$, or $e_1+e_2$ is an element of $H(v)$, and each of these vectors is also an element of $S$. Therefore $S$ also meets every hyperplane. In general the size of $S$ is $\sum_{i=1}^k {n \choose k}$, so to apply Corollary~\ref{mainbinary} we lastly need to impose the restriction that $\sum_{i=1}^k {n \choose k} < 2^{n-2}$. We note that this is equivalent to the inequality
\begin{equation}
	1+ 2 \sum_{i=1}^k {n \choose k} \leq 2^{n-1}
\end{equation}	 
	 which is a more restrictive condition than the inequality given in Equation~\ref{dingbd}.   
	
\end{example}

	We lastly give a tight lower bound on the size of a set of points satisfying the conditions of Theorem~\ref{main}. The following lemma was first proved by Jameson \cite{J}. There are many known proofs of the result; for a survey on them we refer the reader to \cite{B}.

\begin{lemma}\label{bound}
	
	If $S$ is a set of points in $\F_q^n$ meeting every affine hyperplane then $|S| \geq n(q-1)+1$. 
	
\end{lemma}
	
	The lower bound of Lemma~\ref{bound} clearly gives a lower bound on the size of a set of points satisfying the conditions of Theorem~\ref{main}. However, it is not obvious that this bound should be tight since the set of points in Theorem~\ref{main} does not contain the origin. 

\begin{theorem}\label{boundmain}
	
	Let $q$ be an arbitrary prime power. If $S \subseteq \F_q^n \setminus \{0\}$ is a set of points such that

\begin{enumerate}
	\item $S$ is not contained in any affine hyperplane,
	\item $S$ meets every affine hyperplane,
	\item $|S| < q^{n-2}(q-1)$,
\end{enumerate}
	then $|S| \geq n(q-1)+1$. Moreover, this lower bound is tight.
	
\end{theorem}

\begin{proof}
	
	To show that this lower bound is tight, consider the set of points 
\begin{equation*}
	S: =\{ a+ \lambda e_i : \lambda \in \F_q, 1 \leq i \leq n \}
\end{equation*}
	where $a \in \F_q^n \setminus \{0 \}$ is any point not equal to $\lambda e_i$ for any $\lambda \in \F_q$, $1 \leq i \leq n$. By our choice of $a$, the origin is not an element of $S$. 
	Clearly $S$ is not contained in an affine hyperplane, and $|S|=n(q-1)+1 <q^{n-2}(q-1)$. 
	Lastly, if $u_1X_1+u_2X_2+...+u_nX_n=\alpha$ is the equation of an affine hyperplane, then it is easily checked that $a+\lambda e_i$ is a point on the affine hyperplane, where $i$ is chosen such that $u_i \neq 0$, and $\lambda$ is chosen to be the element $\lambda = \frac{1}{u_i}(\alpha- u \cdot a)$. 
	
\end{proof}

The sufficient conditions given in Theorem~\ref{main} and Corollary~\ref{mainbinary} give geometric conditions for the code $\mathcal C_f$ to be minimal and not satisfy the Ashikhmin-Barg condition when $q$ is any prime power and $f$ is an indicator function. Moreover, since the minimum weight of $\mathcal C_f$ is at most $|S|$, then the tight lower bound on the size of a set of points satisfying these conditions given in Theorem~\ref{boundmain} shows that it is possible for $\mathcal C_f$ to additionally have small minimum weight.

\section*{Acknowledgements}

	The author would like to thank Prof. Peter Sin for his sage advice and helpful comments in preparing this paper.

\newcommand{\Addresses}{{
  \bigskip
  \footnotesize

  Julien Sorci, \textsc{Department of Mathematics, University of Florida,
    P. O. Box 118105, Gainesville FL 32611, USA}\par\nopagebreak
  \textit{E-mail address}: \texttt{jsorci@ufl.edu}
	
}}

\maketitle

\Addresses


\begin{thebibliography}{99}

\bibitem{AB} A. Ashikhmin, A. Barg. \emph{Minimal Vectors in Linear Codes}. IEEE Trans. Inf. Theory \textbf{44}(5) (1998) 2010-2017.

\bibitem{B} S. Ball. The polynomial method in Galois geometries. In J. De Beule and L. Storme, editors, \emph{Current research topics in Galois geometry}, Mathematics Research Developments, chapter 5, pages 105–130. Nova Sci. Publ., New York, 2012.

\bibitem{BB} M. Bonini and M. Borello. \emph{Minimal linear codes arising from blocking sets}. arXiv
preprint arXiv:1907.04626, 2019.

\bibitem{CH} S. Chang, J. Y. Hyun. \emph{Linear codes from simplicial complexes}. Des. Codes Cryptogr. DOI: https://link.springer.com/article/10.1007/s10623-017-0442-5 (2017).

\bibitem{DHZ} C. Ding, Z. Heng, Z. Zhou. \emph{Minimal binary linear codes}. IEEE Trans. Inf. Theory, DOI: https://ieeexplore.ieee.org/stamp/stamp.jsp?tp=\&arnumber=8325311\&tag=1.

\bibitem{J} R. Jamison, Covering finite fields with cosets of subspaces, \emph{J. Combin. Theory Ser.} A, \textbf{22} (1977) 253–266.

\bibitem{M} J.L. Massey \emph{Minimal Codewords and Secret Sharing}. In: Proc. 6th Joint Swedish-Russian Int. Workshop on Info. Theory, pp. 276–279 (1993)
	
\end{thebibliography}
\end{document}